\documentclass[12pt]{article}

\usepackage{amssymb,amsmath,amsthm,bm,mathrsfs,mathtools}
\usepackage{setspace}
\usepackage{url,color,bbold,units,pifont}
\usepackage[dvipsnames]{xcolor}
\usepackage[T1]{fontenc}
\usepackage[margin=1in]{geometry}

\usepackage[labelsep=period,labelfont={bf,small},textfont={small},margin=0.8in]{caption}
\RequirePackage[colorlinks,citecolor=blue,urlcolor=blue]{hyperref}

\newcommand{\R}{\mathbb{R}}

\renewcommand{\P}{\mathrm{P}}
\newcommand{\E}{\mathrm{E}}

\renewcommand{\geq}{\geqslant}
\renewcommand{\leq}{\leqslant}

\author{{Robert C. Dalang\footnote{Research partially supported by the Swiss National Foundation for Scientific Research.}}\, \ and {Fei Pu\footnote{Research supported by 
National Natural Science Foundation of China (No. 12201047), Beijing Natural Science Foundation (No. 1232010)
  and the Fundamental Research Funds for the Central Universities.}}
	}

\title{\bf\Large{Hitting with probability one for stochastic heat equations with additive noise}}
\date{}

\newtheorem{stat}{Statement}[section]
\newtheorem{prop}[stat]{Proposition}

\newtheorem{theorem}[stat]{Theorem}

\newtheorem{lemma}[stat]{Lemma}
\theoremstyle{definition}

\newtheorem{remark}[stat]{Remark}

\numberwithin{equation}{section}
\begin{document}
\maketitle
\begin{abstract}
We study the hitting probabilities of the solution to a system of $d$ stochastic heat equations with additive noise subject to Dirichlet boundary conditions. We show that for any bounded Borel set with positive $d-6$-dimensional capacity, the solution visits this set almost surely.
\end{abstract}

\smallskip
\noindent{\it Abbreviated title: Hitting with probability one for  SPDEs with additive noise}

\section{Introduction and main result}

Let $W_i = \{W_i(t, x): (t, x) \in [0, \infty[ \times [0, 1]\}$, $i = 1, \ldots, d$, be independent Brownian sheets defined on a probability space $(\Omega, \mathscr{F}, \mbox{P})$.
Set $W = (W_1, \ldots, W_d)$.
We consider the following system of $d$  stochastic heat equations subject to Dirichlet boundary conditions
 \begin{align}\label{eq1.1}
 \begin{cases}
 \frac{\partial }{\partial t}u(t, x) = \frac{\partial^2 }{\partial x^2}u(t, x)  + \nabla U(u(t,x)) + \frac{\partial^2}{\partial t\partial x}W(t, x), \quad t>0, x\in\, ]0, 1[\\
 u(t,0)= u(t, 1)=0, \quad t>0\\
 u(0, \cdot)= u_0\in C([0, 1], \mathbb{R}^d),
 \end{cases}
 \end{align}
 where $u=(u_1, \ldots, u_d)$, and the function $U: \mathbb{R}^d\mapsto \mathbb{R}$ is differentiable.  We assume that the function $U$ is bounded from above and all the partial derivatives
 $\frac{\partial U}{\partial u_i}$, $i=1, \ldots, d$ of $U$ are globally Lipschitz continuous and bounded.

For $t \geq 0$, let $\mathscr{F}_t = \sigma\{W(s, x), s \in [0, t], x \in [0, 1]\} \vee \mathcal{N}_0$, where $\mathcal{N}_0$ is the $\sigma$-field generated by $\mbox{P}$-null sets.
Let $u_0=(u_0^{(1)}, \ldots, u_0^{(d)})$.
 Following Walsh \cite{Wal86} and Dalang \cite{Dalang1999}, we say that $u$ is a mild solution of \eqref{eq1.1} if $u$ is predictable with respect to $(\mathscr{F}_t)_{t \geq 0}$ and if for $i \in \{1, \ldots, d\}$, $t \in \,]0, \infty[$ and $x \in [0, 1]$, almost surely,
  \begin{align}\label{eq1.2}
  u_i(t, x) &=  \int_0^1 G(t, x, v)u_0^{(i)}(v)dv + \int_0^t\int_0^1 G(t-r, x, v)\frac{\partial U}{\partial u_i}(u(r, v)) dvdr \nonumber\\
  &\quad +\int_0^t\int_0^1 G(t - r, x, v) W_i(dr, dv) ,
 \end{align}
 where the Green kernel $G(t, x, y)$ is given by
 \begin{align}\label{eq2018-02-26-2}
 G(t, x, y) = \sum\limits_{k = 1}\limits^{\infty}e^{-\pi^2k^2t}\phi_k(x)\phi_k(y)
 \end{align}
 with $\phi_k(x) := \sqrt{2}\sin(k\pi x)$; see, for example, \cite{BMS95} and \cite{Wal86}. Under the assumptions on $U$, equation \eqref{eq1.1} admits a unique mild solution; see \cite{Wal86}.

{Dalang, Khoshnevisan and Nualart \cite{DKN07} considers the equations \eqref{eq1.1} with drift term $b(u(t,x))$ instead of $\nabla U(u(t, x)))$ , and they established upper and lower bounds on hitting probabilities of the solution in terms of Hausdorff measure and Newtonian capacity respectively.  In the case $u_0\equiv 0$ and $b\equiv 0$, } according to \cite[Theorems 2.1 and 3.1]{DKN07}, given non-trivial compact intervals $I \subset \, ]0, \infty[$ and $J \subset \,  ]0, 1[$, there exists $c > 0$ depending on $M, I, J$ with $M > 0$, such that for all Borel sets $A \subseteq [-M, M]^d$,
\begin{align*}
c^{-1} \mbox{Cap}_{d - 6}(A) \leq \mbox{P}\{u(I \times J) \cap A \neq \emptyset \} \leq c \mathscr{H}_{d - 6}(A).
\end{align*}

Our goal is to show that if we do not restrict time to a compact interval, but consider the process for all time, then bounded Borel sets with positive $(d - 6)$-dimensional capacity are hit with probability $1$ by the solution to \eqref{eq1.1}. Our main result is the following.

 \begin{theorem}\label{th1.1}
For any bounded Borel set $A \subseteq \mathbb{R}^d$ with positive $(d - 6)$-dimensional capacity, the random field $\{u(t, x)\}_{(t, x) \in [0, \infty[\times [0, 1]}$, starting with any initial value $u_0 \in C([0, 1], \mathbb{R}^d)$, visits this set $A$  almost surely.
\end{theorem}

The random field $\{u(t, x)\}_{(t, x) \in [0, \infty[ \times [0, 1]}$ can also be viewed as a process parameterized only by time and taking values in the space of continuous functions which satisfies the strong Markov property. In Section \ref{section3.2}, we follow \cite{DaZ96, DaZ14, Wal86} to present the canonical Markov system associated with the solution. We will recall from \cite{Fu83} the invariant measure of the solution and deduce a recurrence property based on the law of Brownian bridge. Intuitively, the recurrence property implies that the solution visits the set $A$ infinitely many times  with a positive probability. In Section \ref{section3.4}, we show that the lower bound on hitting probabilities still holds if the solution starts from a non-vanishing initial value, which extends the corresponding results in \cite{DKN07}, and this lower bound is uniform with respect to the initial value, provided the initial value belongs to a fixed ball in the space of continuous functions. We finally give the proof of Theorem~\ref{th1.1} in Section \ref{section3.5}.

 \section{Strong Markov property and invariant measure} \label{section3.2}

 In this section, we present the strong Markov property and the invariant measure  of the solution to \eqref{eq1.1}, which are well known facts in the literature; see Da Prato and Zabczyk \cite{DaZ96, DaZ14} and Funaki \cite{Fu83}. 
 
We first write the equation \eqref{eq1.1} in the abstract form that fits in the framework of \cite{DaZ96, DaZ14}.  Denote the Hilbert space $H=L^2([0, 1], \mathbb{R}^d)$, and $L(H)$ the {space of bounded operators} from $H$ to $H$. We introduce the Sobolev spaces of $\mathbb{R}^d$-valued functions $H^2(]0, 1[, \mathbb{R}^d)$, $H^1_0(]0, 1[, \mathbb{R}^d)$ (see \cite[Section 5.2.2]{Eva98} for the definition). Consider
\begin{align}\label{equation:ab}
\begin{cases}
{d} u(t)= Au(t) + F(u(t)) + {d}W(t), \\
u(0)=u_0\in H,
\end{cases}
\end{align}
where $W$ is a cylindrical Brownian motion on $H$,
 $A$ is the Laplace operator
\begin{align*}
A\xi= \left( \frac{\partial }{\partial x^2}\xi_1, \ldots,  \frac{\partial }{\partial x^2}\xi_d\right), \quad \text{$\xi=(\xi_1, \ldots, \xi_d)\in H^2(]0, 1[, \mathbb{R}^d)\cap H_0^1(]0, 1[, \mathbb{R}^d)$},
\end{align*}
and $F: H\mapsto H$ is given by
\begin{align}\label{F}
(F(\xi))(x) = \nabla U(\xi(x)), \quad \xi\in H, \, x \in [0, 1].
\end{align}
According to Funaki \cite[Appendix I]{Fu83} (see also \cite[Section 4]{DaQ11}\footnote{\cite{DaQ11} considers the equation on the whole space. The treatment there also applies to equation on interval.}),
equations \eqref{eq1.1} and \eqref{equation:ab} are equivalent in the sense  that
the solution $\{u(t, \cdot): t\geq 0\}$ to equations \eqref{eq1.2} also satisfies the following integral equation
\begin{align}\label{equation:abintegral}
u(t)= S(t)u_0 + \int_0^tS(t-s)F(u(s))ds + \int_0^t S(t-s)d W(s),
\end{align}
where $\{S(t): t\geq 0\}$ is the semigroup of Laplace operator given by 
\begin{align*}
(S(t)\phi)(x) = \int_0^1G(t,x, y)\phi(y)d y, \,\  \text{for $\phi\in H$}. 
\end{align*}
Note that the integral equation \eqref{equation:abintegral} is a special case of \cite[(7.29)]{DaZ14} with $B: H\mapsto L(H)$ given by 
\begin{align}\label{B}
(B(u)\xi)(x)= \xi(x), \, x\in [0, 1], \,\, \text{for $u, \xi\in H$}.
\end{align}

We use the notation $u_{u_0}(t, \cdot)$ to indicate that the solution starts from $u_0$. Denote by $\mathcal{B}_b(H)$ the set of bounded Borel functions on $H$. We introduce the transition semigroup of the solution $\{u(t, \cdot): t\geq 0\}$
\begin{align}\label{semigroup}
P_t\phi(u_0):= \mathrm{E}[\phi(u_{u_0}(t, \cdot))], \, \, u_0\in H, \phi\in \mathcal{B}_b(H).
\end{align}
The semigroup $\{S(t): t\geq 0\}$ satisfies condition (iv) of Hypothesis 7.1 of \cite{DaZ96}; see the second equation on page 61 of \cite{DaZ96}. {Moreover, the map $F$ defined in \eqref{F} satisfies condition (ii) of Hyp. 7.1 of \cite{DaZ14}
because the partial derivatives of $U$ are globally Lipschitz and bounded. And the map $B$ defined in \eqref{B}, which maps all elements of $H$ to the identity map on $H$, clearly satisfies (iii) of Hypothesis 7.1 of \cite{DaZ14}.}
Hence, by Theorem 7.1.1 of \cite{DaZ96}, $P_t$ is a strong Feller semigroup for all $t>0$.
Furthermore, Hypothesis 7.2 of \cite{DaZ14} is satisfied (see the calculation in Example 7.6 for the verification of condition (7.27) of \cite{DaZ14}). Hence, the solution has the Markov property, as stated below.

\begin{prop}[{{\cite[Theorem 9.14]{DaZ14}}}]\label{prop2.2}
For $s$, $t \geq 0$, $u_0 \in H$ and $f \in \mathcal{B}_b(H)$, we have
\begin{align}\label{eq2.1}
\mathrm{E}[f(u_{{u_0}}(t+s, \cdot))|\mathscr{F}_s] = P_tf(u_{u_0}(s, \cdot)), \quad \mathrm{P} \,\,\, \mbox{a.s.}
\end{align}
\end{prop}

We proceed to build the canonical Markov system associated to the solution. 
We denote by $E := \{\phi(\cdot) \in C([0, 1], \mathbb{R}^d): \phi(0) = \phi(1) = 0\}$ equipped with the norm
\begin{align*}
\|\phi(\cdot)\|_{\infty} := \sup\limits_{0 \leq v \leq 1}\sup\limits_{1 \leq i \leq d}|\phi_i(v)|.
\end{align*}
As a two-parameter process, the trajectories $(t, x) \mapsto u(t, x)$ are jointly continuous almost surely  (see \cite{Wal86}). Hence almost surely, $t \mapsto u(t, \cdot)$ is continuous in $E$.
Let $\widetilde{\Omega} := C([0, \infty[, E)$ be the space of continuous functions from $[0, \infty[$ to $E$. Denote by $\mathcal{B}(E)$ and $\mathcal{B}_b(E)$ the Borel $\sigma$-field and the set of bounded Borel measurable functions on $E$, respectively. For a generic element $\tilde{\omega} \in \widetilde{\Omega}$, we write $\tilde{\omega}(t, \cdot)$ to indicate the value at $t$, and the second variable appears since $\tilde{\omega}(t, \cdot) \in C([0, 1], \mathbb{R}^d)$. 
For $t \in [0,\infty[$, define $\tilde{u}(t): \tilde{\Omega} \to E$ by
\begin{align*}
\tilde{u}(t) (\tilde{\omega}) = \tilde{\omega}(t, \cdot).
\end{align*}
Define
\begin{align*}
\widetilde{\mathscr{F}}^0_t := \sigma\{\tilde{u}(s): s \leq t\} \quad \mbox{and} \quad \widetilde{\mathscr{F}}^0_{\infty} := \bigvee\limits_{t \geq 0}\widetilde{\mathscr{F}}^0_t.
 \end{align*}
 {We define a family of probability measures $\{\P^{u_0}: u_0\in E\}$ on $(\widetilde{\Omega}, \widetilde{\mathscr{F}}^0_{\infty})$ by}
\begin{align}\label{eq2.4}
    \mbox{P}^{u_0}(A) := \mbox{P}\{u_{u_0} \in A\}, \qquad \mbox{for} \, \, \, u_0 \in E, \ A \in \widetilde{\mathscr{F}}^0_{\infty},
\end{align}
which is determined uniquely by specifying the values on cylindrical sets, i.e.
\begin{align}\label{eq2.5}
  \mbox{P}^{u_0}\{\tilde{u}(t_1) \in B_1, \ldots, \tilde{u}(t_n) \in B_n\}   &= \mbox{P}\{u_{u_0}(t_1, \cdot) \in B_1, \ldots, u_{u_0}(t_n, \cdot) \in B_n\}
\end{align}
for any $n \geq 1$, $B_1, \ldots, B_n \in \mathcal{B}(E)$, and $t_1,  \ldots,  t_n \geq 0$. We denote by $\mbox{E}^{u_0}$ the corresponding expectation with respect to the probability measure $\mbox{P}^{u_0}$.
{We know that $u_0 \mapsto \mbox{P}^{u_0}(A)$ is measurable for sets $A$ of the form $\{\tilde{u}(t)\in B\}$ by \eqref{eq2.4}, \eqref{semigroup} and the strong Feller property of the transition semigroup $\{P_t:t \geq 0\}$.   Using the Markov property \eqref{eq2.1} and induction, the map  $u_0 \mapsto \mbox{P}^{u_0}(A)$ is measurable for sets $A$ of the form $\{\tilde{u}(t_1) \in B_1, \ldots, \tilde{u}(t_n) \in B_n\}$, and hence the measurability also holds for $A \in \widetilde{\mathscr{F}}^0_{\infty}$ by the monotone class theorem. 
Notice from \eqref{eq2.5} that for all $u_0 \in E$,
$P_t f(u_0) = \E^{u_0}[f(\tilde{u}(t)]$.
 Moreover, we deduce from \eqref{eq2.1} that for $f \in \mathcal{B}_b(E)$,
}
\begin{align}\label{eq2.6}
  \mathrm{E}^{u_0}[f(\tilde{u}(t + s))|\widetilde{\mathscr{F}}^0_s] = \mathrm{E}^{\tilde{u}(s)}[f(\tilde{u}(t))] = P_tf(\tilde{u}(s)), \quad \mbox{P}^{u_0} \, \,  \, \mbox{a.s.}
\end{align}
 Let $\mathcal{N}$ be the collection of sets that are $\mbox{P}^{u_0}$-null for every $u_0 \in E$. Define
 \begin{align*}
 \widetilde{\mathscr{F}}_t := \sigma\{\widetilde{\mathscr{F}}^0_t \cup \mathcal{N}\} \quad \mbox{and} \quad \widetilde{\mathscr{F}}_{\infty} := \bigvee\limits_{t \geq 0}\widetilde{\mathscr{F}}_t.
  \end{align*}
  Since the process $\{\tilde{u}(t, \cdot): t \geq 0\}$ has the Markov property (\ref{eq2.6}) and the semigroup $(P_t)_{t \geq 0}$ has the strong Feller property, by Proposition 20.7 of \cite{Bas11}, we  know  that the filtration $(\widetilde{\mathscr{F}}_t)_{t \geq 0}$ is right continuous. Furthermore, we have the following strong Markov property.

\begin{theorem}[Strong Markov property]\label{th2.3}
Suppose $T$ is a finite stopping time with respect to $(\widetilde{\mathscr{F}}_t)_{t \geq 0}$ and $Y$ is bounded and measurable with respect to $\widetilde{\mathscr{F}}_{\infty}$. Then
\begin{align}\label{eq2.7}
\mathrm{E}^{u_0}[Y\circ \theta_T|\widetilde{\mathscr{F}}_T] = \mathrm{E}^{\tilde{u}(T)}[Y],
\end{align}
where $(\theta_t)_{t \geq 0}$ is the shift operator defined by
\begin{align*}
 \theta_t\tilde{\omega}(s, x) := \tilde{\omega}(t + s, x), \qquad \mbox{for} \, \,\, \, \, \tilde{\omega} \in \widetilde{\Omega}, \, \, (s, x) \in [0, \infty[\times [0, 1].
 \end{align*}
\end{theorem}
\begin{proof}
The proof is similar to that of Theorem~20.9 in \cite[p.164]{Bas11}, since only the Feller property of the semigroup $\{P_t: t\geq 0\}$ and the Markov property in \eqref{eq2.6} are needed; see also \cite[Theorem 9.20]{DaZ14}.
\end{proof}

We next introduce the invariant measure of the solution. 
{Let $\mu_0$ be the law on $C([0, 1], \mathbb{R}^d)$ of $d$ independent Brownian bridges on $[0,1]$.}
According to \cite[Theorem 4.1]{Fu83} (see also \cite[Theorem 8.6.3]{DaZ96}), the invariant measure of the solution to \eqref{eq1.1} is given by 
\begin{align}\label{invariant}
\mu(d\phi) =\frac{1}{\rm Z}\exp\left(2\int_0^1U(\phi(x))dx\right)\mu_0(d\phi), \quad \phi\in E,
\end{align}
where $\rm Z$ is a finite positive constant given by
\begin{align*}
{\rm Z}= \int_E  \exp\left(2\int_0^1U(\phi(x))dx\right)\mu_0(d\phi).
\end{align*}
Denote $B(0, R)= \{\phi\in E: \|\phi\|_{\infty}<R\}$, $\bar{B}(0, R)= \{\phi\in E: \|\phi\|_{\infty}\leq R\}$ and $B(0, R)^c= \{\phi\in E: \|\phi\|_{\infty}\geq R\}$.

\begin{lemma}\label{lemma2018-09-14-1}
For all $R > 0$,  $\mu(B(0, R)) > 0$ and $\mu(B(0, R)^{c}) > 0$.
\end{lemma}
\begin{proof}
We first recall the distribution of the supremum of Brownian bridge on $[0, 1]$. Let $B_0(x)= B(x)-xB(1), x\in [0, 1]$, where $\{B(x): x\in [0, 1]\}$ is the standard one dimensional Brownian motion. 
 From \cite[(9.39)]{Bil99}, we have
\begin{align}\label{eq3.15}
    F(R) := \mbox{P}\left\{\sup\limits_{0 \leq x \leq 1}|B_0(x)| \leq R\right\} = 1 + 2\sum\limits_{k = 1}\limits^{\infty}(-1)^ke^{-2k^2R^2},\quad \mbox{for all} \, \,R > 0.
\end{align}
We observe that the distribution function $R\mapsto F(R)$ defined in \eqref{eq3.15} takes values in $]0, 1[$ for all $R > 0$. To see this, from the expression of the alternating series in \eqref{eq3.15}, it is clear that $F(R) < 1$ for all $R > 0$. In order to see that the distribution function $F$ is strictly positive, we apply the triangle inequality,
\begin{align*}
F(R) &\geq \mbox{P}\left\{\sup\limits_{0 \leq x \leq 1}|B(x)| + |B(1)| \leq R\right\} \\
& \geq \mbox{P}\left\{\sup\limits_{0 \leq x \leq 1}|B(x)|  \leq R/2\right\} = H(R/2),
\end{align*}
where
\begin{align*}
H(x) = \frac{4}{\pi}\sum_{k = 0}^{\infty}\frac{(-1)^k}{2k + 1}\exp\left[-\frac{(2k + 1)^2\pi^2}{8x^2}\right], \quad \mbox{for} \,\,\, x > 0
\end{align*}
 denotes the distribution function of the supremum of the absolute value of Brownian motion; see \cite[p.233]{Chu01}. It is clear from the expression of the function $H$ that $H(x) > 0$ for all $x > 0$. Hence we have proved $F(R) > 0$ for all $R>0$.

Fix $R>0$. We have 
\begin{align*}
\mu(B(0, R))&=\frac{1}{\rm Z}\int_{\{\phi\in E:\|\phi\|_{\infty}<R\}} \exp\left(2\int_0^1U(\phi(x))dx\right)\mu_0(d\phi)\\
& \geq \frac{1}{\rm Z}\exp\left(2\inf_{|y|\leq R}U(y)\right)\mu_0(B(0, R))\\
& =\frac{1}{\rm Z} \exp\left(2\inf_{|y|\leq R}U(y)\right) F(R)^d>0.
\end{align*}

We proceed to show that $\mu(B(0, R))^c>0$ for all $R>0$. Taking the derivative of the function defined in \eqref{eq3.15}, 
\begin{align}\label{eq3.15}
    F'(R) =\frac{1}{8R}\sum\limits_{k = 1}\limits^{\infty}(-1)^{k+1}k^2R^2e^{-2k^2R^2},\quad \mbox{for all} \, \,R > 0.
\end{align}
Since the function $r\mapsto re^{-2r}$ is decreasing on $[\frac12, \infty[$, we see that $F'(R)>0$ for all $R>1$.
For fixed $R>0$, we can choose $M>R\vee 1$ and then 
\begin{align*}
\mu(B(0, R)^c)&\geq \mu(B(0, 2M)\setminus B(0, M))\\
&=
\int_{\{\phi\in E: M\leq \|\phi\|_{\infty}<2M\}}  \exp\left(2\int_0^1U(\phi(x))dx\right)\mu_0(d\phi)\\
& \geq \frac{1}{\rm Z} \exp\left(2\inf_{|y|\leq 2M}U(y)\right)  \mu_0(B(0, 2M)\setminus B(0, M))\\
& \geq \frac{1}{\rm Z} \exp\left(2\inf_{|y|\leq 2M}U(y)\right) (F(2M)-F(M))^d>0,
\end{align*}
where the last inequality holds by the mean-value theorem and $F'(R)>0$ for all $R>1$. The proof is complete. 
\end{proof}

We proceed to verify that the transition semigroup $\{P_t: t\geq 0\}$ is irreducible, using Theorem 7.3.1 of \cite{DaZ96}. Notice that Hypothesis 5.1 of \cite{DaZ96} is the same as Hypothesis 7.2 of \cite{DaZ14}. Hence, the conditions of Hypothesis 5.1 of \cite{DaZ96} have been verified (see the discussion above \eqref{eq2.1}). Moreover, since we assume that all the partial derivatives of $U$ are bounded, hence the map $F$ given in \eqref{F} is bounded. Furthermore, because $B(\xi)$ is the identity map for all $\xi\in H$, it clearly meets the condition in Theorem 7.3.1 of \cite{DaZ96}. Hence, the equation \eqref{equation:ab} satisfies all the conditions in Theorem 7.3.1 of \cite{DaZ96}. Hence, we obtain that the transition semigroup $\{P_t: t\geq 0\}$ is irreducible. Therefore, we combine the strong Feller property and irreducible property of the transition semigroup $\{P_t: t\geq 0\}$ with Proposition 4.1.1 and Theorem 4.2.1 of \cite{DaZ96} to conclude that
 for all $R > 0$,
 \begin{align}\label{eq3.17}
    \lim\limits_{t \rightarrow \infty}P_t((u_0, B(0, R)) = \mu(B(0, R)) > 0, \quad \mbox{for all} \, \, \ u_0 \in E.
\end{align}
We refer to \cite[Theorem 11.2.4]{DaZ96} for the corresponding result for one single equation.
{Using the same reason as for \eqref{eq3.17}}, we deduce that
\begin{align}\label{eq3.170}
    \lim\limits_{t \rightarrow \infty}P_t((u_0, B(0, R)^{c}) = \mu(B(0, R)^{c}) > 0, \quad \mbox{for all} \, \, \ u_0 \in E.
\end{align}
We combine \eqref{eq3.17}, \eqref{eq3.170} with Corollary 3.4.6 of \cite{DaZ96} to obtain the following.
\begin{theorem}\label{th3.3}
The Markov process $(\widetilde{\Omega},\widetilde{\mathscr{F}}_t, \tilde{u}(t), \mathrm{P}^{u_0}, P_t, \theta_t)$ is recurrent with respect to $B(0, R)$ and $B(0, R)^c$ for any $R > 0$, i.e., for any $u_0 \in E$,
\begin{align}\label{eq3.18}
   & \mathrm{P}^{u_0}\{\tilde{u}(t) \in B(0, R), \ \mbox{for an unbounded set of t} > 0\} = 1,
\end{align}
and
\begin{align}
   & \mathrm{P}^{u_0}\{\tilde{u}(t) \in B(0, R)^c, \ \mbox{for an unbounded set of t} > 0\} = 1.
\end{align}
\end{theorem}

\section{Lower bound on the hitting probability for solutions with a bounded initial value}\label{section3.4}

{We first recall a general criterion on the lower bound of hitting probabilities for a continuous random field, established in \cite[Theorem 2.1(1)]{DKN07}. }We denote $\Delta$ the parabolic metric defined by $\Delta((t, x); (s, y)) := |t - s|^{1/2} + |x - y|.$
\begin{theorem}[{{\cite[Theorem~2.1(1)]{DKN07}}}] \label{th2017-12-18-1}
 Fix two compact intervals {$I$ and $J$ of $\mathrm{R}$}. Suppose that $\{v(t, x)\}_{(t, x) \in I\times J}$ is a two-parameter continuous random field with values in $\mathbb{R}^d$, such that $(v(t, x), v(s, y))$ has a joint probability density function $p_{t, x; s, y}(\cdot, \cdot)$, for all $s, t\in I$ and $x, y\in J$ with $(t, x) \neq (s, y)$. We denote by $p_{t, x}(\cdot)$ the density function of $v(t, x)$. Assume the following hypotheses:
\begin{enumerate}
  \item[\textbf{A1}] For all $M > 0$, there exists a positive and finite constant $C = C(I, J, M, d)$ such that for all $(t, x) \in I \times J$ and all $z \in [-M, M]^d$,
  \begin{align} \label{eq4.1}
    p_{t, x}(z) \geq C.
  \end{align}

  \item[\textbf{A2}] There exists $\beta > 0$ such that for all $M > 0$, there exists $c = c(I, J, M, d) > 0$ such that for all $s, t \in I$ and $x, y \in J$ with $(t, x) \neq (s, y)$, and for every $z_1, z_2 \in [-M, M]^d$,
     \begin{align}\label{eq4.2}
        p_{t, x; s, y}(z_1, z_2) \leq \frac{c}{[\Delta((t, x); (s, y))]^{\beta/2}}\exp\left(-\frac{\|z_1 - z_2\|^2}{c\Delta((t, x); (s, y))}\right).
     \end{align}
\end{enumerate}
Then the following statement holds.
\begin{enumerate}
  \item[] Fix $M > 0$. There exists a positive and finite constant $a = a(I, J, \beta, M, d)$ such that for all Borel sets $A \subseteq [-M, M]^d$,
  \begin{align}\label{eq4.3}
    \mbox{P}\{v(I \times J) \cap A \neq \emptyset\} \geq a \, \rm{Cap}_{\beta - 6}(A).
  \end{align}
\end{enumerate}
\end{theorem}
\begin{remark}
\begin{itemize}
\item [(1)]The estimate \eqref{eq4.3} in Theorem 2.1 of \cite{DKN07} holds only for compact sets. However, since we can approximate the capacity of a Borel set by the capacity of compact sets as done in the proof of  Proposition 5.2 of \cite{DKN07},  the estimate
\eqref{eq4.3} is also true for Borel sets. 
\item [(2)] The constant $a$ in \eqref{eq4.3} does not depend on the particular random field $v$, but is determined from the constants $C$ and $c$ in  \eqref{eq4.1} and \eqref{eq4.2} as well as $I, J, \beta, M$ and $d$.
\end{itemize}
\end{remark}

Theorem \ref{th2017-12-18-1} applies to the solution to equations \eqref{eq1.1} ($u_0\equiv 0$, $b\equiv 0$) subject to Neumann boundary conditions,  with $I=[t_0, T]$ ($t_0>0$), $J=[0, 1]$ and $\beta = d$; see \cite[Proposition 4.1]{DKN07}. 
It also applies to the case of Dirichlet boundary conditions with $J$ be away from boundaries (see \cite[Remark 4.7]{DKN07}).
 We will apply Theorem  \ref{th2017-12-18-1}  to derive a uniform lower bound on hitting probabilities for the solution $\{u(t, x)\}_{(t, x) \in I\times J}$ to \eqref{eq1.1}, where the constant will not depend on the specific choice of initial data $u_0\in \bar{B}(0, N)$, but only on  $N$.
 
\begin{lemma}\label{lemma2018-09-14-2}
 {Fix $M, N > 0$, $\epsilon>0$ and two compact intervals $I \subset ]0, \infty[\,, J\subset ]0, 1[$}.  There exists a finite positive constant $a = a(I, J, N, M, d, \epsilon)$ such that for all Borel sets $A \subseteq [-M, M]^d$, and for all $u_0 \in \bar{B}(0, N)$,
  \begin{align}\label{eq4.10}
    \mathrm{P}\{u_{u_0}(I \times J) \cap A \neq \emptyset\} \geq a\, \left(\rm{Cap}_{d - 6}(A)\right)^{1+\epsilon},
  \end{align}
and equivalently,
 \begin{align}\label{eq4.11}
    \mathrm{P}^{u_0}\{\tilde{u}(I \times J) \cap A \neq \emptyset\} \geq a\, \left(\rm{Cap}_{d - 6}(A)\right)^{1+\epsilon}.
  \end{align}
\end{lemma}
\begin{proof}
The statement \eqref{eq4.11}  is a consequence of \eqref{eq4.10} and  (\ref{eq2.4}) and we only need to prove \eqref{eq4.10}.
We first prove that in the case that the drift term vanishes, namely, $\nabla U\equiv 0$, 
there exists a finite positive constant $K = K(I, J, N, M, d)$ such that for all Borel sets $A \subseteq [-M, M]^d$, and for all $u_0 \in \bar{B}(0, N)$,
  \begin{align}\label{eq4.102}
    \mathrm{P}\{u_{u_0}(I \times J) \cap A \neq \emptyset\} \geq K\, \mbox{Cap}_{d - 6}(A),
  \end{align}
and then apply Girsanov’s theorem to deduce \eqref{eq4.10} from \eqref{eq4.102}.

Assume $\nabla U\equiv 0$.
In order to prove \eqref{eq4.102}, by Theorem \ref{th2017-12-18-1}, it suffices to prove that hypotheses \textbf{A1} and \textbf{A2} are satisfied for $\{u(t, x)\}_{(t, x) \in I\times J}$, where the constants depend additionally on $N$, but not on the specific choice of $u_0$. 
For $(t, x)\in [0, \infty[\, \times [0, 1]$, let $v(t, x) = (v_1(t, x), \ldots, v_d(t, x))$ and $\lambda(t, x) = (\lambda_1(t, x), \ldots, \lambda_d(t, x))$, where
\begin{align}
  v_i(t, x) &= \int_0^t\int_0^1 G(t - r, x, v) W_i(dr, dv), \label{eq:v}\\
  \lambda_i(t, x) &= \int_0^1 G(t, x, v)u_0^i(v)dv. \label{eq2018-08-23-1}
\end{align}
We add the superscripts $u$ or $v$  to the probability density functions to indicate to which random field they correspond.

{\em Verification of \textbf{A1}}. Fix $M > 0$ and let $z \in [-M, M]^d$. Then for all $(t, x) \in I \times J$, the probability density function of $u(t, x)$ is given by
\begin{align}\label{eq4.4}
    p_{t, x}^u(z) = \frac{1}{(2\pi \sigma_{t, x}^2)^{d/2}}\exp\left(-\frac{\|z - \lambda(t, x)\|^2}{2\sigma_{t, x}^2}\right),
\end{align}
where
\begin{align}\label{eq4.5}
    \sigma_{t, x}^2 := \mbox{Var}(u_i(t, x)) = \int_0^tdr\int_0^1dv \, (G(t - r, x, v))^2.
\end{align}
Since $(t, x) \mapsto v_i(t, x)$ is $L^2$-continuous by (4.11) of \cite{DKN07}, it follows that the function $(t, x) \mapsto \sigma_{t, x}$ achieves its minimum $\rho_1 > 0$ and its maximum $\rho_2 < \infty$ over $I \times J$. Moreover, since $u_0 \in \bar{B}(0, N)$, from \eqref{eq2018-08-23-1}, it is clear that $\sup_{(t, x) \in I \times J}\|\lambda(t, x)\| \leq c_N$ for some constant $c_N$. Thus, for all $(t, x) \in I \times J$ and all $z \in [-M, M]^d$,
\begin{align*}
    p_{t, x}^u(z) \geq \frac{1}{(2\pi \rho_2^2)^{d/2}}\exp\left(-\frac{(M^2 + c_N^2)d}{2\rho_1^2}\right),
\end{align*}
which proves \textbf{A1}.

{\em Verification of \textbf{A2}}. First, we give some estimates on the regularity of the function $(t, x) \mapsto \lambda(t, x)$ on $I \times J$. Denote $t_0=\inf\{t: t\in I\}>0$.

{\em Case 1: $s = t$, $x \neq y$}. For all $t \in I$, $x, y \in J$ and $1 \leq i \leq d$,
\begin{align}
  |\lambda_i(t, x) - \lambda_i(t, y)| &=  \left|\int_0^1\sum\limits_{k = 1}\limits^{\infty}e^{-\pi^2k^2t}(\phi_k(x) - \phi_k(y))\phi_k(v)u_0^i(v)dv\right| \nonumber\\
    &\leq  2N\sum\limits_{k = 1}\limits^{\infty}e^{-\pi^2k^2t}|\sin(k\pi x) - \sin(k\pi y)| \nonumber\\
    &\leq  2N\pi \sum\limits_{k = 1}\limits^{\infty}ke^{-\pi^2k^2t_0}|x - y| \leq c_1N|x - y|.  \label{eq3.11}
\end{align}

{\em Case  2: $x = y$, $s < t$}. For all $t$, $s \in I$,
\begin{align}
   |\lambda_i(t, x) - \lambda_i(s, x)| &= \left|\int_0^1\sum\limits_{k = 1}\limits^{\infty}(e^{-\pi^2k^2t} - e^{-\pi^2k^2s})\phi_k(x)\phi_k(v)u_0^i(v)dv\right|  \nonumber\\
    &\leq \int_0^1\sum\limits_{k = 1}\limits^{\infty}e^{-\pi^2k^2s}|1 - e^{-\pi^2k^2(t - s)}|\, |\phi_k(x)\phi_k(v)u_0^i(v)|dv  \nonumber\\
    &\leq 2N\sum\limits_{k = 1}\limits^{\infty}e^{-\pi^2k^2t_0}|1 - e^{-\pi^2k^2(t - s)}| \nonumber\\
    &\leq 2N\pi^2(t-s)\sum\limits_{k = 1}\limits^{\infty}k^2e^{-\pi^2k^2t_0} \leq c_2N(t-s),\label{eq4.7} 
    \end{align}
    where the third inequality is due to the fact that $1-e^{-x}\leq x$ for all $x\geq0$.

Hence, \eqref{eq3.11} and \eqref{eq4.7} together imply that there exists a constant $\tilde{c}$ such that for all $s, t \in I$, $x, y \in J$, $1 \leq i \leq d$,
\begin{align}\label{eq4.8}
    (\lambda_i(t, x) - \lambda_i(s, y))^2 \leq \tilde{c}\, N^2((t - s)^2 + (x - y)^2).
\end{align}

Using the upper bound on the joint probability density function of $(v(t, x), v(s, y))$ (see of \cite[(2.3), Theorem 4.6 and {Remark 4.7}]{DKN07}), the elementary equality $(a - b)^2 \geq \frac{1}{2}a^2 - b^2$ and (\ref{eq4.8}), we obtain that
\begin{align}
   p_{t, x; s, y}^u(z_1, z_2) &=  p_{t, x; s, y}^v(z_1 - \lambda(t, x), z_2 - \lambda(s, y)) \nonumber\\
    &\leq \frac{c}{[\Delta((t, x); (s, y))]^{d/2}}\exp\left(-\frac{\|z_1 - \lambda(t, x) - z_2 + \lambda(s, y)\|^2}{c\, \Delta((t, x); (s, y))}\right)  \nonumber\\
    &\leq \frac{c}{[\Delta((t, x); (s, y))]^{d/2}}\exp\left(-\frac{\frac{1}{2}\|z_1 - z_2\|^2 - \|\lambda(t, x) - \lambda(s, y)\|^2}{c\, \Delta((t, x); (s, y))}\right)  \nonumber\\
    &\leq  \frac{\tilde{c}_N}{[\Delta((t, x); (s, y))]^{d/2}}\exp\left(-\frac{\|z_1 - z_2\|^2}{2c\,\Delta((t, x); (s, y))}\right),
\end{align}
which proves \textbf{A2}.

Therefore, in the case $\nabla U\equiv 0$, the lower bound on hitting probabilities in \eqref{eq4.102} follows from the result of Theorem \ref{th2017-12-18-1}.

For the general drift term, denote $b=(b_1, \ldots, b_d):=\nabla U$.  Girsanov’s theorem (see \cite[(5.4)--(5.6)]{DKN07}) yields that 
\begin{align*}
  \mathrm{P}\{u_{u_0}(I \times J) \cap A \neq \emptyset\} = \mathrm{E}[\bm{1}_{\{[(v+\lambda)(I \times J)] \cap A \neq \emptyset\}}J_t^{-1}],
\end{align*}
where $v$ and $\lambda$ are defined in \eqref{eq:v} and  \eqref{eq2018-08-23-1} and
\begin{align*}
J_t&=\exp\bigg(-\int_0^t\int_0^1G(t-s,x, y) b((v+\lambda)(s, y))\cdot W(ds, dy)\\
& \qquad \qquad\qquad
+\frac12\int_0^t\int_0^1G^2(t-s, x, y)\|b((v+\lambda)(s, y))\|^2dsdy
\bigg).
\end{align*}
We apply H\"older's inequality (see \cite[(5.5) and (5.6)]{DKN07}) to obtain that for $\epsilon>0$
\begin{align}\label{lower_uniform}
  \mathrm{P}\{u_{u_0}(I \times J) \cap A \neq \emptyset\} &\geq    
  \left(\mathrm{P}\{[(v+\lambda)(I \times J)] \cap A \neq \emptyset\}\right)^{1+ \epsilon} 
  \left(\mathrm{E}\left[J_t^{1/\epsilon}\right]\right)^{-\epsilon}\nonumber\\
  & \geq \left(K\, \mbox{Cap}_{d - 6}(A)\right)^{1+\epsilon}\left(\mathrm{E}\left[J_t^{1/\epsilon}\right]\right)^{-\epsilon},
  \end{align}
where the second inequality holds by \eqref{eq4.102}. Since the functions $b_i, i=1, \ldots, d$ are bounded, we can use the same argument as in \cite[(5.7)]{DKN07} to see that there exists a constant $C(\epsilon)>0$ depending only on $\epsilon$ such that
\begin{align}\label{J_t}
\left(\mathrm{E}\left[J_t^{1/\epsilon}\right]\right)^{-\epsilon} \geq C(\epsilon).
\end{align}
Therefore, the uniform lower bound on hitting probabilities in \eqref{eq4.10} follows from \eqref{lower_uniform} and \eqref{J_t}. 
\end{proof}

\section{Proof of Theorem~\ref{th1.1}}\label{section3.5}

We first state a result on hitting probabilities for general Markov processes, which will be used to prove  Theorem \ref{th1.1}.

\begin{prop}\label{prop2017-12-18-1}
Let $(\Omega, \mathscr{F}_t, X(t), \theta_t, \mathrm{P}^{x})_{t \geq 0, x \in \mathcal{E}}$ be a continuous Markov system taking values in the Banach space $\mathcal{E}$, which has the strong Markov property.  Fix $K > N > 0$ and a Borel set $\mathcal{A} \subset \mathcal{E}$. Suppose that the process $\{X(t): t \geq 0\}$ is recurrent with respect to $B(0, N)$ and $B(0, K)^{c}$, and that there exists a positive constant $c = c(N, K, \mathcal{A})$ such that for all $u_0 \in \bar{B}(0, N)$,
\begin{align}\label{eq2017-12-18-1}
\mathrm{P}^{u_0}\{\exists\, t \in [0, T_1], \, \, \, \mbox{s.t.}  \, \, \,  X(t) \in \mathcal{A}\} \geq c,
\end{align}
where $T_1 := \inf\{t \geq 0: \|X(t)\| > K\}$. Then for any $x \in \bar{B}(0, N)$,
\begin{align}\label{eq2017-12-18-2}
\mathrm{P}^{u_0}\{\exists \, t \geq 0, \, \, \, \mbox{s.t.}  \, \, \,  X(t) \in \mathcal{A}\} = 1.
\end{align}
\end{prop}
\begin{proof}
Due to (4.1), we have $c \leq 1$, and there is nothing to prove if $c=1$, so we assume $0 < c < 1$.
By the recurrence property, we define inductively two sequences of finite stopping times $(T_k)_{k \geq 0}$ and $(S_k)_{k \geq 0}$, as follows. Let $S_0 = 0$, and for $k \geq 1$,
\begin{align*}
    T_k = \inf\{t \geq S_{k - 1}: \|X(t)\| > K\}, \quad  S_k = \inf\{t \geq T_k: \|X(t)\| < N\},
\end{align*}
which satisfy that $T_k = S_{k - 1} + T_1\circ\theta_{S_{k - 1}}$.

For each $k \geq 1$, we set $\mathcal{A}_k := \{\exists \, t \in [S_{k - 1}, T_k] ,\ \mbox{s.t.}\ X(t) \in \mathcal{A}\}$. By the strong Markov property,
\begin{align}
    \mbox{P}^{u_0}\{\mathcal{A}_k|\mathscr{F}_{S_{k - 1}}\}
    &=  \mbox{P}^{u_0}\{\exists \, t \in [S_{k - 1}, S_{k - 1} + T_1\circ\theta_{S_{k - 1}}],\ \mbox{s.t.}\ X(t) \in \mathcal{A} |\mathscr{F}_{S_{k - 1}}\} \nonumber\\
    &=  \mbox{P}^{u_0}\{\exists \, t \in [0, T_1\circ\theta_{S_{k - 1}}] ,\ \mbox{s.t.}\ X(t + S_{k - 1}) \in \mathcal{A} |\mathscr{F}_{S_{k - 1}}\} \nonumber\\
    &=  \mbox{E}^{u_0}[\bm{1}_{\{\exists \, t \in [0, T_1],\ \mbox{s.t.}\ X(t) \in \mathcal{A}\}}\circ\theta_{S_{k - 1}} |\mathscr{F}_{S_{k - 1}}] \nonumber \\
    &= \mbox{E}^{X(S_{k - 1})}[\bm{1}_{\{\exists\,  t \in [0, T_1],\ \mbox{s.t.}\ X(t) \in \mathcal{A}\}}] \nonumber\\
    &\geq c,
    \end{align}
where the inequality is due to \eqref{eq2017-12-18-1}. Therefore, for any integer $n \geq 1$,
\begin{align}
 &\mbox{P}^{u_0}\{\exists\,  t \geq 0, \, \, \, \mbox{s.t.}  \, \, \,  X(t) \in \mathcal{A}\} \nonumber \\
    &\quad \geq  \mbox{P}^{u_0}\left(\bigcup\limits_{k = 1}\limits^{n}\mathcal{A}_k\right) =  1 - \mbox{P}^{x}\left(\bigcap\limits_{k = 1}\limits^{n}\mathcal{A}_k^c\right) \nonumber\\
    &\quad=  1 - \mbox{E}^{u_0}\left[\left(\prod\limits_{k = 1}\limits^{n - 1}\bm{1}_{\mathcal{A}_k^c}\right)\mbox{P}^{x}\{\mathcal{A}_n^c|\mathscr{F}_{S_{n - 1}}\}\right] \nonumber\\
    &\quad=  1 - \mbox{E}^{u_0}\left[\left(\prod\limits_{k = 1}\limits^{n - 1}\bm{1}_{\mathcal{A}_k^c}\right)(1 - \mbox{P}^{x}\{\mathcal{A}_n|\mathscr{F}_{S_{n - 1}}\})\right] \nonumber\\
    &\quad\geq  1 - (1 - c)\mbox{E}^{u_0}\left[\prod\limits_{k = 1}\limits^{n - 1}\bm{1}_{\mathcal{A}_k^c}\right] \geq  1 - (1 - c)^n, \label{eq5.4}
\end{align}
where we repeat the argument to get the last inequality. Letting $n \rightarrow \infty$, we obtain \eqref{eq2017-12-18-2}.
\end{proof}

\begin{proof}[Proof of Theorem \ref{th1.1}]
Let ${\rm Cap}_{d-6}(A)>0$.
We assume that $u_0 \in \bar{B}(0, N)$ and $A \subseteq [-N, N]^d$ for some $N > 0$. 
Recall $v_i$ and $\lambda_i$ in \eqref{eq:v} and \eqref{eq2018-08-23-1}.
First,
we give an estimate on the following tail probability. Using \eqref{eq1.2}, we see that  for any $K > N+2 \|\nabla U\|_{\infty}$ (here, $ \|\nabla U\|_{\infty}:=\sup_{1\leq i\leq d}\sup_{y\in \R^d}|\frac{\partial U}{\partial u_i}(y)|$),
\begin{align}
 \mbox{P}\left\{\sup\limits_{0 \leq t \leq 2}\|u_{u_0}(t, \cdot)\|_{\infty} \geq K\right\}   &\leq  \sum\limits_{i = 1}\limits^{d}\mbox{P}\left\{\sup\limits_{0 \leq t \leq 2}\sup\limits_{0 \leq x \leq 1}|u_i(t, x)| \geq K\right\} \nonumber\\
    &\leq \sum\limits_{i = 1}\limits^{d}\mbox{P}\left\{2\|\nabla U\|_{\infty}+\sup\limits_{0 \leq t \leq 2}\sup\limits_{0 \leq x \leq 1}|v_i(t, x) + \lambda_i(t, x)| \geq K\right\}  \nonumber\\
    &\leq  d \, \mbox{P}\left\{\sup\limits_{0 \leq t \leq 2}\sup\limits_{0 \leq x \leq 1}|v_i(t, x)| \geq K - N- 2\|\nabla U\|_{\infty}\right\}. \label{eq5.1}
\end{align}
Since $(t, x) \mapsto v_i(t, x)$ is continuous almost surely, we have
\begin{align*}
    \lim\limits_{K \rightarrow \infty}\mbox{P}\left\{\sup\limits_{0 \leq t \leq 2}\sup\limits_{0 \leq x \leq 1}|v_i(t, x)| \geq K - N- 2\|\nabla U\|_{\infty}\right\} = 0.
\end{align*}
Therefore, letting $a$ be the constant in \eqref{eq4.10}, where $I = [1, 2]$ and {$J \subset ]0, 1[$ is a fixed compact interval}, we can choose sufficiently large $K>N$, depending only on $N$, $\|\nabla U\|_\infty$, $a$ and ${\rm Cap}_{d-6}(A)$, such that
\begin{align*}
   \mbox{P}\left\{\sup\limits_{0 \leq t \leq 2}\|u_{u_0}(t, \cdot)\|_{\infty} \geq K\right\} \leq 2^{-1-\epsilon}a \, \left(\mbox{Cap}_{d - 6}(A)\right)^{1+\epsilon},
\end{align*}
or, equivalently, for all $u_0$ with  $\|u_0\|_{\infty} \leq N$,
\begin{align}\label{eq5.2}
   \mbox{P}^{u_0}\left\{\sup\limits_{0 \leq t \leq 2}\|\tilde{u}(t)\|_{\infty} \geq K\right\} \leq 2^{-1-\epsilon}a\, \left(\mbox{Cap}_{d - 6}(A)\right)^{1+\epsilon}.
\end{align}

By \cite[(5.10)]{DKN07}, we choose a compact set $A' \subset A$ such that
\begin{align}\label{eq2018-09-13-2}
\mbox{Cap}_{d - 6}(A') \geq \frac{2}{3}\mbox{Cap}_{d - 6}(A).
\end{align}
Define \begin{align*}
\mathcal{A}:= \{\phi(\cdot) \in E: \exists\, x \in [0, 1] \, \, \, \mbox{s.t.} \, \, \, \phi(x) \in A'\}.
\end{align*}
It is straightforward to check that $\mathcal{A}$ is a closed (hence Borel) subset of $E$. Notice that
\begin{align} \label{eq2018-09-13-1}
 \mbox{P}\{u_{u_0}([0, \infty[ \times [0, 1]) \cap A \neq \emptyset\} &\geq  \mbox{P}\{u_{u_0}([0, \infty[ \times [0, 1]) \cap A' \neq \emptyset\}\nonumber \\
    &= \mbox{P}^{u_0}\{\tilde{u}([0, \infty[ \times [0, 1]) \cap A' \neq \emptyset\} \nonumber\\
    &= \mbox{P}^{u_0}\{\exists \, t \geq 0, \,\, \mbox{s.t.}\, \, \tilde{u}(t) \in \mathcal{A}\}.
\end{align}
In order to prove Theorem \ref{th1.1}, by \eqref{eq2018-09-13-1}, Proposition \ref{prop2017-12-18-1} and Theorem \ref{th3.3}, it suffices to verify that the estimate for the hitting probability in \eqref{eq2017-12-18-1} holds for the process $\{\tilde{u}(t): t \geq 0\}$, {with $T_1$ given by $T_1=\inf\{t\geq 0: \|\tilde{u}(t)\|_{\infty}>K\}$}. This is indeed the case, since
 \begin{align}
   & \mbox{P}^{u_0}\{\exists \,  t \in [0, T_1] ,\,\, \mbox{s.t.} \, \, \tilde{u}(t) \in \mathcal{A}\} \nonumber\\
    &\quad =  \mbox{P}^{u_0}\{\exists \,  (t, x) \in [0, T_1] \times [0, 1],\ \, \mbox{s.t.}\, \, \tilde{u}(t, x) \in A'\} \nonumber\\
    &\quad \geq  \mbox{P}^{u_0}\left\{\{\exists \,  (t, x) \in [0, T_1] \times [0, 1],\, \, \mbox{s.t.}\, \, \tilde{u}(t, x) \in A'\}\cap\left\{\sup\limits_{0 \leq t \leq 2}\|\tilde{u}(t)\|_{\infty} \leq K \right\}\right\} \nonumber\\
    &\quad \geq  \mbox{P}^{u_0}\left\{\{\exists \,  (t, x) \in [1, 2] \times J,\ \, \mbox{s.t.}\, \, \tilde{u}(t, x) \in A'\}\cap\left\{\sup\limits_{0 \leq t \leq 2}\|\tilde{u}(t)\|_{\infty} \leq K \right\}\right\}   \nonumber\\
    &\quad \geq \mbox{P}^{u_0}\{\exists \,  (t, x) \in [1, 2] \times J,\  \, \mbox{s.t.}\,\, \tilde{u}(t, x) \in A'\}  - \mbox{P}^{u_0}\left\{\sup\limits_{0 \leq t \leq 2}\|\tilde{u}(t)\|_{\infty} \geq K \right\} \nonumber\\
    &\quad \geq a\, \left(\mbox{Cap}_{d - 6}(A') \right)^{1+\epsilon}- 2^{-1-\epsilon}a \,\left(\mbox{Cap}_{d - 6}(A)\right)^{1+\epsilon} \geq  ((2/3)^{1+\epsilon}-2^{-1-\epsilon})a\, \left(\mbox{Cap}_{d - 6}(A)\right)^{1+\epsilon}, \nonumber
\end{align}
where the fourth inequality follows from \eqref{eq4.11} and \eqref{eq5.2} and the last inequality is due to \eqref{eq2018-09-13-2}. 
Hence \eqref{eq2017-12-18-1} is satisfied and this completes the proof of Theorem \ref{th1.1}.
\end{proof}

\begin{remark}
We expect that the conclusion of Theorem \ref{th1.1} remains valid for the solution to  the nonlinear system of stochastic heat equations
\begin{align*}
\frac{\partial u_i}{\partial t}(t, x) &=\frac{\partial^2 u_i}{\partial x^2}(t, x) +\sum_{j = 1}^d\sigma_{ij}(u(t, x))\dot{W}^j(t, x) + b_i(u(t, x))
\end{align*}
with Dirichlet boundary conditions and initial condition $u_0$, where the functions $b_i$ and $\sigma_{ij}$ satisfy hypotheses \textbf{P1} and \textbf{P2} of \cite{DKN09}. However, proving this requires several facts which seem not to be available in the literature.
\begin{itemize}
  \item [(1)] The uniform lower bound in Lemma \ref{lemma2018-09-14-2}.  {The lower bound in \cite[Theorem 1.2(a)]{DKN09} is in terms of $d-6+\eta$-dimensional capacity and  for the solution with vanishing initial conditions. 
 This lower bound was improved in terms of $d-6$-dimensional capacity in \cite{DaP21}; see Theorem 1.3(a) and Remark 1.4 of \cite{DaP21}.
  Following through the proof in \cite{DKN09, DaP21}, this could be extended to other initial conditions, which would yield the uniform bound in Theorem \ref{th2017-12-18-1} needed for Lemma \ref{lemma2018-09-14-2}.}
  \item [(2)] The strong Markov property of Theorem \ref{th2.3} is available; see \cite[Theorem 9.20]{DaZ14}.
  \item [(3)] The recurrence property in Theorem \ref{th3.3} would be needed. This could be deduced from the existence and properties of an invariant measure as above. Existence and uniqueness of an invariant measure are available in \cite{Cer01, LiR15, Sta11}, but the required properties of this invariant measure seem not to be available.
\end{itemize}

\end{remark}

\noindent\textbf{Acknowledgement}.  Part of this paper is based on F. Pu's Ph.D. thesis \cite{Pu18}, written under the supervision of R. C. Dalang.

\vskip1cm
\begin{small}
\noindent\textbf{Robert C. Dalang}
Institut de Math\'ematiques, Ecole Polytechnique
F\'ed\'erale de Lausanne, Station 8, CH-1015 Lausanne,
Switzerland.\\
Email: \texttt{robert.dalang@epfl.ch}\\
\end{small}

\begin{small}
\noindent\textbf{Fei Pu}
Laboratory of Mathematics and Complex Systems,
School of Mathematical Sciences, Beijing Normal University, 100875, Beijing, China.\\
Email: \texttt{fei.pu@bnu.edu.cn}\\
\end{small}
\end{document}